\documentclass[11pt]{article}
\usepackage{setspace}
\usepackage{amsfonts}
\usepackage{amsmath}
\usepackage{hyperref}
\usepackage{amsthm}
\usepackage{amssymb}
\usepackage[top=1.25in, bottom=1.25in, left=1.in, right=1.in]{geometry}

\newtheorem{theorem}{Theorem}

\newtheorem{lemma}[theorem]{Lemma}

\newtheorem{claim}[theorem]{Claim}

\newcommand{\Addresses}{{% additional braces for segregating \footnotesize
  \bigskip
  \footnotesize
  
  Julian Sahasrabudhe, \textsc{Instituto Matem\'{a}tica Pura e Aplicada (IMPA), Rio de Janeiro, Brasil 
 \emph{and} Peterhouse, Cambridge, UK.}\par\nopagebreak
  \textit{julian.sahasra@gmail.com}
}}

\begin{document}
  
\def\N{\mathbb{N}}
\def\Z{\mathbb{Z}}
\def\t{\theta}
\def\eps{\varepsilon}

\title{ Counting Zeros of Cosine Polynomials: On a Problem of Littlewood}  

\author{Julian Sahasrabudhe}

\maketitle
\begin{abstract} 
We show that if $A$ is a finite set of non-negative integers then the number of zeros of the function  
\[ f_A(\theta) = \sum_{a \in A } \cos(a\theta) ,
\] in $[0,2\pi]$, is at least $(\log \log \log |A|)^{1/2-\varepsilon}$. This gives the first unconditional lower bound on a problem of Littlewood and solves a conjecture of Borwein, Erd\'{e}lyi, Ferguson and Lockhart. We also prove a result that applies to more general cosine polynomials with few distinct rational coefficients. One of the main ingredients in the proof is perhaps of independent interest: we show that if $f$ is an exponential polynomial with few distinct integer coefficients and $f$ ``correlates'' with a low-degree exponential polynomial $P$, then $f$ has a very particular structure. This result allows us to prove a structure theorem for trigonometric polynomials with few zeros in $[0,2\pi]$ and restricted coefficents.
\end{abstract}

\onehalfspace

\section{Introduction} 
We call a sum of the form
\[ \sum_{a = 0 }^N C_a \cos(a\theta), \]
a \emph{cosine polynomial}, where $\theta \in [0,2\pi]$, $C_a \in \mathbb{R}$ and $N \in \mathbb{N}$. Sums of this type have been a central topic of study in mathematics since and before the foundational work of Joseph Fourier in 1807, who made the remarkable discovery that any (nice enough) periodic function can be realized as the superposition of a countable number of sine-waves. Since these early results, the rich interplay between the behavior of the function $f$ and its Fourier coefficients $C_1,\ldots,C_N$ have received considerable attention and have found numerous applications throughout and beyond mathematics. 
\paragraph{}
The starting point of a rich field of study begins with the following question: What can be said about the behavior of a cosine polynomial if we impose restrictions on its coefficients? Or, similarly, what can be said about the behaviour of a function $f$ if we restrict its Fourier transform? Of course, if we impose \emph{no} conditions on the coefficients $C_0,\ldots,C_N$, essentially nothing can be said. Indeed, it is well known that any continuous, even function $f: \mathbb{R}/(2\pi\mathbb{Z}) \rightarrow \mathbb{R}$ can be uniformly approximated by cosine polynomials. Thus, for large $N$, the behavior of such polynomials is essentially as rich as the class of continuous functions. In what follows, we consider what can be said about the behavior of the function $f$ if the coefficients are restricted to a small set. 
\paragraph{}
One of the early champions of questions about polynomials with restricted coefficients was J.E. Littlewood, who studied many problems of this general type and stimulated the field with a wealth of problems \cite{JELw1,JELw2,JELw3,JELw4}. Perhaps most famously, Littlewood asked for the minimum $L_1$-norm attained by a function of the form $\sum_{a\in A} \cos(a\theta)$, where $A \subseteq \mathbb{N} \cup \{0\}$ is of a given size. This question was resolved (up to constants) in 1981 independently by Konyagin \cite{Konyagin}, and McGehee, Pigno, and Smith \cite{MPS} after a series of partial results were obtained by Salem, Cohen, Davenport and Pichorides \cite{Salem, Cohen, Davenport,Pich}. We note that these estimates have seen numerous applications and modifications in recent years. 
%Similar $L_1$-estimates have also been obtained in the discrete setting of $\mathbb{Z}/p\mathbb{Z}$ by Green and Konyagin \cite{GreenKonyagin} and applied to combinatorial problems.
\paragraph{}
Counting zeros of polynomials with restricted coefficients has also received considerable attention; after the seminal work of Littlewood and Offord \cite{LitOff}, several results have been proved on the number of roots of ``typical'' polynomials with restricted coefficients. For example, Kac \cite{Kac1,Kac2} has given an exact integral formula for the expected number of roots of a degree $N$ polynomial with coefficients sampled identically and independently from a Gaussian distribution. In a similar vein, Erd\H{o}s and Offord \cite{ErdosOfford} showed that almost all polynomials of the form $\sum_{i=0}^N \varepsilon_i x^i$, where $\varepsilon_0,\ldots, \varepsilon_N \in \{-1,+1\}$, have $(\frac{2}{\pi} + o(1)) \log N$ real roots. These results have been subsequently refined and extended by various authors. For more, we refer the reader to the recent paper of Hoi Nguyen, Oanh Nguyen and Van Vu \cite{VVu}. 
\paragraph{}
The study of extremal questions on the number of real zeros of polynomials with restricted coefficients began with the work of Bloch and P\'{o}lya \cite{BlochPolya} in 1932, who studied the maximum number of real zeros attainable by polynomials of the form $\sum_{i=0}^N \varepsilon_i x^i$, $\varepsilon_0,\ldots, \varepsilon_N \in \{-1,0,+1\}$. After improvements by other authors \cite{Schur,Szego}, Erd\H{o}s and Tur\'{a}n proved a landmark result in the study of roots of polynomials. They showed that any polynomial with sufficiently ``flat'' coefficients will have roots that are radially ``equidistributed'' in the complex plane \cite{ErdosTuran}. This result has seen numerous extensions, for which we refer the interested reader to the survey of Granville \cite{AGran}.  Odlyzko and Poonen \cite{OP} have proved several results on the geometric structure of the zero set in $\mathbb{C}$ of polynomials with $0,1$ coefficients.\ Questions about the number of zeros in polygons and annuli in the complex plane have also been studied \cite{BE1,BEL}.
\subsection{Main results}
The problem that we study in this paper has its origins in an old question of Littlewood on cosine polynomials with $0,1$ coefficients. For a finite set $A \subseteq \mathbb{N} \cup \{0\}$, put $N = |A|$ and let 
\[ f_A(\theta) = \sum_{a \in A } \cos(a\theta) ,
\] for $\theta \in [0,2\pi]$. In his 1968 monograph \cite{JELw4}, Littlewood asked for the lower bound on the number of zeros of the function $f_A(\theta)$ in terms of $N$. He conjectured that the number of such zeros is ``possibly $N-1$ or not much less''. In 2008, Borwein, Erd\'{e}lyi, Ferguson and Lockhart \cite{BEFL} disproved this conjecture by showing the existence of cosine polynomials, of the above form, with at most $CN^{5/6}\log N$ roots in a period. They went on to conjecture that the number of zeros of $f_A$ tends to infinity as $N$ tends to infinity.
\paragraph{}
Progress on Littlewood's question of \emph{lower bounds} was achieved by Borwein and Erd\'{e}lyi \cite{BElowerbounds} who showed that if $A \subseteq \mathbb{N} \cup \{0\}$ is an \emph{infinite} set and $A_n = A \cap [1,n]$, for each $n \in \mathbb{N}$, then the number of zeros of $f_{A_n}$ tends to infinity with $n$. Interestingly, their argument is both ineffective and depends crucially on the nested structure of the sets $A_n$ and thus does not imply that the roots of a general sequence of polynomials tends to infinity. In this paper we answer the question of Borwein, Erd\'{e}lyi, Ferguson and Lockhart in a strong form: by giving an explicit lower bound on the number of zeros for Littlewood's problem.
\begin{theorem} \label{thm:main}
For a finite set $A \subseteq \mathbb{N} \cup \{0\}$, the cosine polynomial 
\[ f_A(\theta) = \sum_{a \in A } \cos(a\theta) , \] has at least $(\log \log \log |A|)^{1/2 + o(1)}$ distinct roots in $[0,2\pi]$. \end{theorem}

\paragraph{}
We prove Theorem~\ref{thm:main} by way of a more general result. For a finite set of real numbers $R$, define $M(R) = \max\{ |x| : x \in R \}$. We take all logarithms to be in \emph{base 2} and for $r \in \mathbb{N}$ we adopt the notation $\log_{(r)}$ for the $r$th-iterated logarithm.

\begin{theorem} \label{thm:main2} For $n \in \N$, and finite $R \subseteq \Z$, there exists constants $C, C_R  > 0 $ so that the following holds. 
If \[ f(\theta) = \sum_{r = 0}^n C_r\cos(r\theta),
\] is a cosine polynomial with $C_r \in R$, for all $r \in [0,n]$, and $|f(0)| \geq C(\log M(R))^2$ then, $f$ has at least $C_R \left( \frac{ \log_{(3)} |f(0)| }{\log_{(5)} |f(0)| }\right)^{1/2}$ zeros in $[0,2\pi]$.
\end{theorem}

Perhaps it is curious that our bound on the number of zeros depends on $|f(0)|$ (or equivalently $\left| \sum_{r} C_r \right|$) rather than on the number of non-zero terms. However, as we will 
briefly touch on in subsection~\ref{subsec:StructureThm}, a bound depending on the number of non-zero terms is too much to ask for; we really are required to use a more subtle parameter (like $|f(0)|$) in our lower bounds.

\paragraph{}
The first main step in our proof of Theorem~\ref{thm:main2} is to show that every exponential polynomial (with coefficients from a fixed finite, rational set) that ``correlates'' with a low-degree exponential polynomial $P$ (with general coefficients) must have a very particular structure. Putting it more precisely, if $f$ has coefficients from a finite set $R \subseteq \mathbb{Z}$ and there exists a ``low-degree'' exponential polynomial $P$ for which $\left| \int Pf \right| \geq \varepsilon \int |Pf|$, then $f$ can be written as a sum of a bounded number of exponential polynomials $f = f_1 + \cdots + f_L$, where each of the $\hat{f}_i(r)$ are supported on integer intervals $I_i$ and $(\hat{f}_i(r))_{r \in I_i}$ are periodic with period bounded in terms of $\varepsilon, P, R$ only. To state our result precisely, we let $R \subseteq \Z$ be a finite set, $k \in \Z$, and $\eps >0$. We define the functions $\beta_0(R,k,\eps)$, $\beta_1(k)$ to be
\[ \beta_0(R,k,\eps) = 2^{\eps^{-1}|R|^{(Ck^2\log\log k )}} \]
\[ \beta_1(k) = Ck\log\log k, 
\] where $C$ is a sufficiently large constant. 
\begin{theorem}  \label{thm:CorWithPolyThm}
For $\varepsilon >0$, $k,n \in \mathbb{N}$, let $R \subseteq \mathbb{Z}$ be a finite set and let $f$ be an exponential polynomial with $\hat{f}(r) \in R$, for all $r \in \mathbb{Z}$.
If there exists a non-zero, exponential polynomial $P$ with $\deg(P) \leq k$ such that 
\[ \left| \int Pf \right| \geq \varepsilon \int |Pf|, 
\] then there exists a positive integer $L \leq \beta_1(R,k,\eps) $ 
exponential polynomials $Q_1,\ldots,Q_L$, with $\deg(Q_j) \leq \beta_1(k)$, for all $j \in [L]$, and 
$N_1,\ldots, ,N_L,M_1,\ldots,M_L \in \Z$ so that 
\[ f(\theta) = \sum_{j=1}^L Q_j(\theta)\left(\frac{ e^{iN_j\theta} - e^{iM_j\theta} }{1 - e^{i (p_j+1) \theta}} \right). 
\] Here the $o(1) = o_k(1)$ terms tend to zero as $k$ tends to infinity. \end{theorem}

Once we have established Theorem~\ref{thm:CorWithPolyThm}, it will be easy to see that a potential counter-example to our Theorem~\ref{thm:main2} must have quite a bit of structure. The second main step in our proof is to deduce that all appropriate cosine polynomials with this structure must have many zeros.

\subsection{A structure theorem for trigonometric polynomials with few roots} \label{subsec:StructureThm}
Perhaps surprisingly, there are trigonometric polynomials that look suspiciously similar to those in Theorem~\ref{thm:main} but have a \emph{bounded} number of zeros.
For example, the simply described sine polynomial 
\[ f_n(\t) = 2\sin(\t) +  \sum_{r = 1}^{2n} \sin((2r+1)\t),  \] 
has exactly 2 zeros for $\t \in [0,2\pi]$. It is not hard to get at the crux of what is going on: while the sine polynomial $S_n(\t) = \sum_{r=1}^n \sin(r\t)$ has many zeros itself, the sign changes with $\t \approx 0 $ are not very dramatic. Indeed $S_n(\t) > -\t$ as $\t$ approaches zero from the right. Hence adding $\sin(\t)$
to this sum (to form $\tilde{f}_n = 2\sin(\t) + \sum_{r=2}^n \sin( r\t) $) kills all of these zeros with $\t \approx 0$. Now, $\tilde{f}_n(\t)$ still has many zeros with $\theta \approx \pi$, but this small wrinkle can be eliminated 
by defining 
\[ f_n(\t) = \frac{1}{2}\left(\tilde{f}_{2n}(\t) + \tilde{f}_{2n}(\pi - \t)\right), \]
thus bringing us to $f_n$, the example mentioned above.

While it might seem that this situation is limited to trigonometric series involving sine, one can easily symmetrize this example
to produce \emph{cosine} series with few zeros. Indeed, the cosine series
\[ g_n(\t) = 2\cos(\t) +  \sum_{r=2}^{2n} \sin(r\pi/2)\cos(r\t),
\]  with coefficents in $\{0,1,-1,2\}$, is easily seen to have only $2$ zeros. This example is really the same as the previous for 
\[ g_n(\t) = \frac{1}{2}\left(f_n(\t + \pi/2) + f_n(-\t + \pi/2)\right).
\] While numerous other examples of trigonometric polynomials with few coefficients can be obtained in a similar way,
our Theorem~\ref{thm:CorWithPolyThm} will allow us to deduce that the highly periodic structure is no coincidence; every real trigonometric polynomial with few zeros \emph{must} have this periodic structure. We state the theorem here without explicit bounds on the relevant constants. However, if the reader is interested, these can (essentially) be read off from Theorem~\ref{thm:CorWithPolyThm}. 
\begin{theorem}
Let $R \subseteq \Z$, $|R| < \infty$ and $K \in \N$, there exists $C = C(K,R) \in \N$ so that the following holds. If 
\[ f(\theta) = \sum_{r = 1}^n A_r\cos(r\t) + \sum_{r=1}^n B_r\sin(r\t), 
,\] is a trigonometric polynomial with at most $K$ roots and with $A_r,B_r \in R$, for all $r$, then there exists an integer $L < C$ and disjoint intervals $I_1, \ldots, I_L$, with $\bigcup_{i=0}^L I_i = [n]$
so that for all $i \in [L]$, the sequences $(A_r)_{r \in I_i}$ and $(B_r)_{r \in I_i}$ are periodic with period at most $C$. \end{theorem}
\paragraph{}
The example $g_n(\t)$ is also interesting for another reason; it shows that a trigonometric polynomial with many non-zero coefficients may have few zeros, even among cosine polynomials with coefficients restricted to a finite set of integers. Note however, that this example does not contradict our Theorem~\ref{thm:main2}, as $|g_n(0)|$ remains bounded as $n$ tends to infinity. 
\paragraph{}
Our results also make progress towards answering the related question for cosine polynomials of the type $\sum_{a = 0}^N \varepsilon_a \cos(a\theta)$, where $\varepsilon_0,\ldots,\varepsilon_N \in \{-1,+1\}$. This is often stated as an equivalent question about ``reciprocal polynomials'' \cite{BEL}. In this setting, it is also believed that the number of zeros tends to infinity with the degree \cite{Mukunda}. So far, several partial results have been obtained. First, Erd\'{e}lyi \cite{Erdelyi} observed that a polynomial of this form has at least one real zero. This was later strengthened by Mukunda \cite{Mukunda}, who showed that odd degree polynomials of this type have at least $3$ roots. Drungilas \cite{Drun} extended this result by showing that such polynomials of degree $d$ have at least $5$ roots if $d \geq 7$ and even, while such polynomials have at least $4$ roots if $d \geq 14$ and odd.  Our results show that potential counter-example to the conjecture for coefficients in $\{-1,+1\}$ must have a very particular periodic form, as in Theorem~\ref{thm:CorWithPolyThm}
\subsection{Organization and notes}
Our paper is organized as follows. In Section~\ref{sec:preliminaries} we prove some basic results which are of a linear-algebraic nature. In Section~\ref{sec:MainLemmas} we state the celebrated solution to the Littlewood Conjecture regarding $L_1$-norms of exponential polynomials, a result which plays a key role in our proof. We also prove several key lemmas. In Section~\ref{sec:reduction} we prove Theorem~\ref{thm:CorWithPolyThm}, our structural result on exponential polynomials which correlate with low-degree polynomials and our first main step towards Theorem~\ref{thm:main2}. We also note how this result allows us to deduce a strong structural result for cosine polynomials having few zeros. In Section~\ref{sec:theStructuredCase} we show that cosine polynomials which admit this structure must have many roots, our second main step. Section~\ref{sec:ProofsOfThms} houses the final deduction of Theorem~\ref{thm:main2} by using the various results built up in previous sections.
\paragraph{}
This result was originally announced in April 2015 as a part of the nomination announcement for the Ralph Faudree award at the University of Memphis and was first presented in a lecture at the Institute for Advanced Studies (IMT), Lucca, Italy, in a seminar on May 23rd of 2015.  Recently it has been been brought to our attention that Erd\'{e}lyi \cite{TEFit} has shown that if $f_n$ is a sequence of cosine polynomials with coefficients from a finite set of real numbers and $|f_n(0)|$ tends to infinity then the number of roots of $f_n$ tends to infinity. While Erd\'{e}lyi's argument is ineffective and gives no concrete bound on the number of zeros of any given cosine polynomial, it gives another proof of the conjecture of Borwein, 
Erd\'{e}lyi, Ferguson and Lockhart. Erd\'{e}lyi's paper first appeared as a pre-print in February of 2016.

\section{Preliminaries}~\label{sec:preliminaries}
For integers $n<m$, let $[n,m] = \{ n, n+1,\ldots, m \}$ and if $n \geq 1 $, put $[n] = \{ 1,\ldots, n\}$. As is standard, for an exponential polynomial $f : \mathbb{R}/ ( 2 \pi \mathbb{Z} ) \rightarrow \mathbb{C}$, we define the \emph{Fourier coefficients} of $f$ to be the sequence $\left(\hat{f}(r)\right)_{r \in \mathbb{Z}}$, where 
\[ \hat{f}(r) = \frac{1}{2\pi} \int_{0}^{2\pi} f(\theta)e^{-ir\theta} d\theta,
\] for $r \in \mathbb{Z}$. If $f$ and $g$ are exponential polynomials, it is trivial to check \emph{Parseval's formula} and the \emph{inversion formula}. That is,  
\[ \frac{1}{2\pi} \int_{0}^{2\pi} f(\theta)\overline{g(\theta)} = \sum_{r \in \mathbb{Z}} \hat{f}(r)\overline{\hat{g}(r)}\ \text{    and }\] 
\[ f(\theta) = \sum_{r \in \mathbb{Z}} \hat{f}(r)e^{ir\theta}.
\]For $N \in \mathbb{N}$, $d \in [N]$ and $x \in \mathbb{C}^N$, define the collection of $d$-\emph{windows} of $x$ to be the set of vectors 
\[ \{ (x(1+r),\ldots,x(d+r)) \}_{r =0}^{N-d} \subseteq \mathbb{C}^d. \] We refer to the $\mathbb{C}$-vector space spanned by the $d$-windows of $x$ as \emph{the space of} $d$-\emph{windows of} $x$. In the course of the proof, we shall proceed differently depending on whether or not certain spaces of $d$-windows have full rank. 
\paragraph{}
For the case that the space of $d$-windows of $x$ fails to have full rank, we write $(x(n))_{n=1}^N$ in the form
\begin{equation} \label{equ:fromofx(n)} x(n) = \sum_{i =1}^l \rho_i^nQ_i(n),
\end{equation} for all $n \in [N]$, where $\rho_1,\ldots,\rho_l$, $l \in [d]$ are appropriate complex numbers and $Q_i(n)$ are polynomials with bounded degree. Of course, this is very similar to the well-known results on the space of solutions to a linear recurrence relation. Here we quickly note a finitary version of this result. 

For $\rho \in \mathbb{C}\setminus \{0\}$, $N \in \mathbb{N}$ and $k \in \mathbb{N}\cup \{0\}$, we define the vector $b_N(\rho,k) \in \mathbb{C}^N$ by  $b_N(\rho,k) = (\rho^nn^k)_{n=1}^N$. For $m \in \mathbb{N}$, we define $S_N(\rho;m) = \{b_N(\rho,0),\ldots,b_N(\rho,m-1) \}$. We note the following fact about the linear independence of these vectors. 
\begin{lemma} \label{lem:LinInd} Let $l, m_1,\ldots,m_l \in \mathbb{N}$ and $\rho_1,\ldots,\rho_l \in \mathbb{C}\setminus \{0\}$. If $N\geq \sum_{i=1}^l m_i$ then 
$S_N(\rho_1,m_1) \cup \cdots \cup S_N(\rho_l,m_l)$ is a linearly independent collection over $\mathbb{C}$. \qed 
\end{lemma}
From this, we deduce the following. 
\begin{lemma} \label{lem:effectiveLinearRecurrenceRelation}
For $N,d \in \mathbb{N}$ with $N \geq 3d$, let $v = (A_0,\ldots,A_{d-1}) \in \mathbb{C}^N$ be non-zero and let $\rho_1,\ldots,\rho_l \in \mathbb{C}$ be the distinct non-zero roots of $P(X) = A_0 + \cdots + A_{d-1}X^{d-1}$ with multiplicities $m_1,\ldots,m_l \in \mathbb{N}$. If all of the $d$-windows of $x = (x(n))_{n=1}^N \in \mathbb{C}^{N}$ are orthogonal to $v$ then there exists a unique choice of polynomials $Q_1(X),\ldots,Q_l(X) \in \mathbb{C}[X]$ with $\deg(Q_i) \leq m_i - 1$,  $i \in [l]$, such that
\begin{equation} \label{equ:effLemFormofx} x(n) = \sum_{i= 1}^l \rho_i^nQ_i(n),
\end{equation} for all $n \in [d,N-d]$. 
\end{lemma}
\begin{proof}
We express $v = (0,\ldots,0, A_a,\ldots,A_b,0,\ldots,0)$, where $a,b \in [0,d-1]$ and $A_a,A_b\not=0$. Also write $w = (A_a,\ldots,A_b)$, set $t = b - a+1$, $M = N-2d+1$ and let $V$ be the space of sequences $x \in \mathbb{C}^{[d,N-d]}$ for which $w$ is orthogonal to the space of $t$ windows of $x$. We see that $V$ can be expressed as the kernel of a triangular matrix with $A_a \not=0$ on the diagonal and $M - t + 1$ rows. Therefore $\dim(V) = t-1$. Now, it is easy to check that the set of $t-1 = \sum_i m_i$ vectors $(\rho_i^nn^j)_{n=d}^{N-d}$, $i \in [l]$, $j \in [0,m_i-1]$, lie in the space $V$. By the lower bound on $N$, we have $M = N-2d+1 \geq t-1 =  \sum_{i \in [l]} m_i$. We apply Lemma~\ref{lem:LinInd} to learn that these sequences are linearly independent and therefore, since there are $t-1$ of them, they form a basis for the space $V$. Since $(x(n))_{n= d}^{N-d} \in V$, $(x(n))_{n= d}^{N-d}$ has a unique expression in this basis. This completes the proof.  
\end{proof}

We record some properties of a linear map that shall appear later in the proof of Theorem~\ref{thm:CorWithPolyThm}. For $N \in \mathbb{N}$, $p \in [N-1]$ and $\rho \in \mathbb{C}$, we define the \emph{differencing operator} $\Delta_{\rho,p} : \mathbb{C}[X]\rightarrow \mathbb{C}[X]$, by 
\begin{equation} \label{equ:TwistedDiffop} (\Delta_{\rho, p} Q)(X) = \rho^pQ(X+p) - Q(X).
\end{equation} Also, for $d \in \mathbb{N}$, let $\mathbb{C}_d[X]$ be the space of polynomials with degree at most $d$. 

\begin{lemma} \label{fact:KernelOfDiffOp}
For $\rho \in \mathbb{C}$ and $p \in \mathbb{N}$, $\Delta_{\rho,p}$ defines a linear map $\Delta_{\rho,p} : \mathbb{C}_{d}[X] \mapsto \mathbb{C}_d[X]$
with the following properties. 
\begin{enumerate}
\item $\Delta_{\rho,p}$ has non-trivial kernel if and only if $\rho$ is a $p$th root of unity.
\item If $\rho$ is a $p$th root of unity, then the kernel is precisely the space of constant functions. 
\end{enumerate}
\end{lemma}
\begin{proof}
First note that $\Delta = \Delta_{\rho,p}$ is a linear map from $\mathbb{C}_d[X]$ to $\mathbb{C}_d[X]$. For the rest, we observe how this map acts on the natural basis $\{1,X,\ldots X^d\}$. For $k \in [0,d]$ we have
\begin{equation} \label{equ:DeltaMap}
 \Delta X^k = \rho^p(X+p)^k - X^k  = (\rho^p-1)X^k + k\rho^{p}X^{k-1} + \rho^p\sum_{i = 0 }^{k-2} \binom{k}{i}p^{n-i}X^i. 
\end{equation}Thus if we express the map $\Delta$ as a matrix $A$ with respect to the basis $1,X,\ldots,X^d$, the resulting matrix is in triangular form. If $\rho$ is \emph{not} a $p$th root of unity, then (\ref{equ:DeltaMap})  all the diagonal entries of $A$ are non-zero. Thus $A$ has full rank and $\Delta$ has trivial kernel. If $\rho$ \emph{is} a $p$th root of unity, (\ref{equ:DeltaMap}) tells us that the diagonal is all zero, while the subdiagonal is nowhere zero. Thus the kernel has dimension $1$. In this case, it is clear that the space of constant functions is in the kernel. This completes the proof. \end{proof}

%We also require a basic fact that applies to sequences for which the space of $d$-windows has rank less than $d$ and with finitely many distinct coefficients. The proof follows immediately from an application of the pigeon hole principle.

%\begin{lemma} \label{fact:PeriodicPigeonHole}
%Let $N,d \in \mathbb{N}$, let $R \subseteq \mathbb{C}$ and let $x = (x(n))_{n=1}^N \in R^N$. If $N > |R|^d + d$ and the space of $d$-windows of $x$ has rank strictly less than $d$, then $x$ is periodic with period at most $|R|^d$. \qed
%\end{lemma}

\section{Some Results and Prepatory Lemmas} \label{sec:MainLemmas}
An important tool for our results is the ``solution to the Littlewood conjecture'', regarding the minimum of a $L_1$ norm of an exponential polynomial. After progress by several authors \cite{Salem, Cohen, Davenport,Pich}, Littlewood's ``$L_1$-norm conjecture'' was proved independently by Konyagin \cite{Konyagin}, and McGehee, Pigno, and Smith \cite{MPS}. We actually use a strong form of the result proved by McGehee, Pigno, and Smith \cite{MPS}, which we state here. 

\begin{theorem} \label{thm:LittlewoodConj} Let $N \in \mathbb{N}$, $A = \{ a_1 < a_2 < \cdots < a_N \} \subseteq \mathbb{Z}$ and
\[ f(\theta) = \sum_{r =1}^N C_r e^{ia_r\theta} 
\] be an exponential polynomial, where $C_1,\ldots,C_N \in \mathbb{C}$. Then there exists a function $h = h(f)$ such that $|h|_{\infty} \leq 1$ and 
\[ \int fh \geq \frac{1}{60} \sum_{r = 1}^N \frac{|C_r|}{r}.  \]
\end{theorem}
Of course, this result immediately implies Littlewood's conjecture, that if $A \subseteq \mathbb{N} \cup \{0\}$ is a finite set then the polynomial 
$f_A(\theta) = \sum_{a \in A} \cos(a\theta)$ has $L_1$ norm at least $c\log |A|$. We remark that we do not require the full strength of the result above in the sense that a weaker quantitative form of Theorem~\ref{thm:LittlewoodConj} would still yield the results we desire, albeit with weaker bounds. 
\paragraph{}
The following sequence of lemmas play a key role in our proof of Theorem~\ref{thm:CorWithPolyThm} and our main step towards Theorem~\ref{thm:main2}. The following result says that if $Ax = b$, where $A$ is a matrix with full rank and integer entries, then if $b$ is ``small'', $x$ must also be ``small''. We ultimately apply this result when we encounter a segment of our coefficient sequence whose corresponding space of $d$-windows has full rank. 

\begin{lemma} \label{lem:InversionBound} 
For $n \in \mathbb{N}$, let $A$ be an $n\times n$ invertible matrix with entries from $R \subseteq \mathbb{Z}$. If $Ax = b$ for some $x,b \in \mathbb{C}^n$, then $|x|_{\infty} \leq M(R)^{n-1}n^{n/2}|b|_{\infty}$. 
\end{lemma}
\begin{proof}
If we write $x = (x_1,\ldots,x_n)$,  Cramer's theorem tells us that $x_i = \frac{\det(A_i)}{\det(A)}$, for $i \in [n]$, where $A_i$ is the matrix $A$ with the $i$th column replaced with the vector $b$. We have that
\[ |\det(A_i)| \leq (M(R)n^{1/2})^{n-1}|b|_{2} \leq M(R)^{n-1}n^{n/2}|b|_{\infty},
\] as $|\det(A_i)|$ is bound by the product of the $\ell_2$-norms of the column vectors of $A_i$, $i \in [n]$. We have also used $|b|_2 \leq n^{1/2}|b|_{\infty}$.
\paragraph{} On the other hand, since $A$ is an invertible matrix we have that $\det(A) \not=0$. Therefore $|\det(A)| \geq 1$, as $A$ is a matrix with integer entries and the determinant is a polynomial in these entries. Therefore $|x_i| \leq M(R)^{n-1}n^{n/2}|b|_{\infty}$ for all $i \in [n]$, as desired. \end{proof}
%Turning to the general case; if $R \subseteq \frac{1}{l}\mathbb{Z}$, we have the equation $(lA)x = lb$, where $lA$ is a matrix with integer entries. We apply the work above to the integer matrix $(lA)$ and vector $lb$. We have that the entries of $A$ are at most $lM(R)$ and $|bl|_{\infty} = l|b|_{\infty}$. Hence we arrive at the desired bound. 

The next lemma, in contrast with the previous lemma, is ultimately applied when certain spaces of $d$-windows \emph{fail} to have full rank. In this case we can derive quite a bit of information about the structure of the original sequence. We also use this lemma to derive some properties of a class of exponential polynomials. For $k \in \mathbb{N}$, we define the exponential polynomial $S_k$ by
\begin{equation} \label{def:DefOfSk} S_k(\theta) = \prod_{r = 1}^k \left( 1 - e^{ir\theta} \right) .
\end{equation} Note that $\deg(S_k) = k(k+1)/2 \leq k^2$. In what follows, we show that if the space of $d$-windows of a rational sequence of coefficients has rank $<d$, then this sequence can be expressed as the sum of periodic sequences with bounded period. In our arguments, it is crucial that this bound on the period depends only on $d$ and not on $|R|$. 

\begin{lemma} \label{lem:RankOfWindows}
For $N,d,t \in \mathbb{N}$, let $R \subseteq \mathbb{Q}$ be a finite set such that $N > |R|^d + 3d$, and let $x = (x(r))_{r=1}^N \in R^N$ be a sequence. Denote by $V$ the space of $d$-windows of $x$ and write $x' = (x(n))_{n=d}^{N-d}$.
\begin{enumerate}
\item \label{item:RankOfWindows1} If $\dim(V) = t < d$ then $x'$ can be written as $x' = x_1 + \cdots + x_t$, where each
$x_i \in \mathbb{C}^{[d,N-d]}$, $i \in [t]$, is periodic with period at most $16t\log\log(t+3)$. 

\item \label{item:RankOfWindows2} For $k \in \mathbb{N}$, let $S_k(\theta)$ be as above, let $n = \deg(S^2_k) = k(k+1)$ and write $S^2_k(\theta) = B_0 + B_1e^{i\theta} + \cdots + B_{n}e^{n i\theta}$. If $d \geq k(k+1) + 1$ and $v = (B_0,\ldots,B_n,0,\ldots,0) \in  V^{\perp}$ then we may write $x' = x_1 + \cdots + x_n$ where each $x_i \in \mathbb{C}^{[d,N-d]}$, $i \in [n]$, is periodic with period at most $k$. 
\end{enumerate}
\end{lemma}

We prove Lemma~\ref{lem:RankOfWindows}, by first proving that if $x = (x(n))_{n=1}^N$ is a sequence of the form (\ref{equ:fromofx(n)}) that is restricted to a finite set, then $x'$ can be expressed a linear combination of terms of the form $\rho^n$, where $\rho$ is a root of unity, \emph{and} the root of a certain bounded-degree polynomial.

\begin{lemma} \label{lem:ExpressAsRoots} Let $R \subseteq \mathbb{C}$ be a finite set, let $N,t \in \mathbb{N}$ satisfy $N \geq |R|^{(t+1)} + 3(t+1)$, and let $x = (x(r))_{r=1}^N \in R^N$ be a sequence. If $V$ is the space of $(t+1)$-windows of $x$ and if $v = (C_0,\ldots,C_t) \in V^{\perp}$ is non-zero then there exists $\alpha_1,\ldots,\alpha_l \in \mathbb{C}$ so that $x(r) = \alpha_1\rho_1^r + \cdots + \alpha_l\rho_l^r $ for all $r \in [t+1,N-(t+1)]$, where $l \in [t]$ and $\rho_1, \ldots, \rho_l$ are roots of unity and zeros of the polynomial $P(X) = C_0 + C_1X + \cdots + C_tX^t$.
\end{lemma}
\begin{proof}
Let $\rho_1,\ldots, \rho_l$, $l \in [t]$, be the distinct, non-zero complex roots of the polynomial $P(X) = C_0 + C_1X + \cdots + C_tX^t$ with multiplicities $m_1,\ldots,m_l$, respectively. Put $I = [t+1,N-(t+1)]$ and apply Lemma~\ref{lem:effectiveLinearRecurrenceRelation} to write
\begin{equation} \label{equ:Expandx(r)} x(r) = \sum_{i = 1}^l \rho_i^rQ_i(r)  ,
\end{equation} for $r \in I$, where $Q_1(X),\ldots,Q_l(X)$ are polynomials over $\mathbb{C}$ with $\deg(Q_i)\leq m_i -1$ for $i \in [l]$. As $|I| > N - 2(t+1) \geq |R|^{(t+1)} + (t+1)$, there must exist an integer $p$ and an interval $I' \subseteq I$ so that $I'+p \subseteq I$, $|I'| \geq t+1$ and $x(r) = x(r+p)$, for all $r \in I'$. Thus, for all $r \in I'$, we have 
\begin{align*}
 0 &= x(r+p) - x(r) \\
   &= \sum_{i=1}^l \rho_i^r(\rho_i^p Q_i(r+p) - Q_i(r))  \\
   &=  \sum_{i=1}^l \rho_i^r(\Delta_{\rho_i,p}Q_i)(r) ,
 \end{align*} where $\Delta_{\rho,p}$ is the differencing operator defined at (\ref{equ:TwistedDiffop}). Now since $|I'| \geq t+1$ the vectors $(\rho_i^rr^j)_{r \in I}$, $j \in [0,m_i-1]$ for $i \in [l]$, are linearly independent, by Lemma~\ref{lem:LinInd}. It then follows that the polynomials $\Delta_{\rho_1,p}Q_1,\ldots,\ \Delta_{\rho_l,p}Q_l \in \mathbb{C}[X]$ are all identically zero. We would like to use this to say something about the polynomials $Q_1,\ldots,Q_l$. 
\paragraph{}
Let $S  = \{ i \in [l]: \rho_i^p = 1 \}$. If $i \not\in S$, we apply Lemma~\ref{fact:KernelOfDiffOp} to learn that $\Delta_{\rho_1,p}$ has trivial kernel and therefore $Q_i(X) \equiv 0$. If $i \in S$, then $\rho_i$ \emph{is} a $p$th root of unity, and we apply Lemma~\ref{fact:KernelOfDiffOp} to learn that $Q_i(X) = \alpha_i$, for some $\alpha_i \in \mathbb{C}$. So, from line (\ref{equ:Expandx(r)}), we have
\[ x(r) = \sum_{i \in S} \alpha_i\rho_i^r , 
\]for all $r \in I = [t+1,N-(t+1)]$, as desired. \end{proof}

We now use Lemma~\ref{lem:ExpressAsRoots} to give the proof of Lemma~\ref{lem:RankOfWindows}. In what follows, we let $\phi$ denote the Euler totient function. That is, if $n \in \mathbb{N}$, $\phi(n)$ is the number of positive integers $x \leq n$ that are co-prime to $n$. It is well known that $\phi(n)$ grows approximately as $cn/(\log\log n)$. We use the explicit lower bound $\phi(n) \geq n/(8\log\log n)$, which holds for all integers $n > 3$ \cite{EulerBound}. 
\paragraph{}
\emph{ Proof of Lemma~\ref{lem:RankOfWindows} :} Let $V$ denote the space of $d$-windows of $x$. To prove Conclusion~(\ref{item:RankOfWindows1}), note that we have $\dim(V^{\perp}) = d - t $. Since $V$ is spanned by vectors with rational entries, it follows that we may find a basis $\mathcal{B}$ for $V^{\perp}$ using rational vectors. We then may find a non-zero vector $v \in V^{\perp} \cap \mathbb{Q}^d$ which is supported on the first $t+1$ coordinates: let $B$ be a $(d-t) \times d$ matrix whose rows are the vectors of $\mathcal{B}$; then use Gaussian elimination on $B$ to produce a matrix which is lower-triangular form. The rows of this resulting matrix will still be in $V^{\perp} \cap \mathbb{Q}^d$ and will include a vector $v$ which is supported on the coordinates with indexes in $[t+1]$.

Now write $v = (C_0,\ldots,C_t,0, \ldots ,0 ) \in \mathbb{Q}^d$ and let $\rho_1,\ldots, \rho_l$, $l \in [t]$, be the distinct, non-zero complex roots of the polynomial $P(X) = C_0 + C_1X + \cdots + C_tX^t$. Since $N \geq |R|^{t+1} + 3(t+1)$, we may apply Lemma~\ref{lem:ExpressAsRoots} to write
\[ x(r) = \sum_{i=1}^l \alpha_i \rho^r_i, 
\] for $r \in [t+1,N-(t+1)] \supseteq [d,N-d]$, where $\rho_1,\ldots,\rho_l$ are roots of unity and roots of the polynomial $P(X)$. We take $(x_i(r))_{r = d}^{N-d}  = (\alpha_i\rho_i^r)_{r=d}^{N-d} $ for each $i \in [l]$.
\paragraph{}
We now show that we have the desired bound on the periods of the sequences $x_1,\ldots,x_l$. For this, let $i \in S$. We know that $\rho_i$ is a root of unity and so $\rho_i$ is a $q_i$th \emph{primitive} root of unity, for some $q_i \in \mathbb{N}$. It is well known that the $q_i$th cyclotomic polynomial $\Phi_{q_i}(X)$ is the minimal polynomial of $\rho_i$ (over the field $\mathbb{Q}$) and has degree $\phi(q_i)$. As $\rho_i$ is also a root of $P(X) \in \mathbb{Q}[X]$, it follows that $\deg(\Phi_{q_i}(X)) \leq \deg(P(X))$ and hence $\phi(q_i) \leq \deg(P(X)) \leq t$. If $q_i \leq 16$ we trivially have $q_i \leq 16 \leq  16t\log(\log(t+3))$, as $t \geq 1$. Otherwise, if $q_i > 16$ we have 
\[ \frac{q_i}{8\log\log q_i} \leq \phi(q_i) \leq  t ,\] which implies $q_i \leq 16t\log\log(t+3)$. Therefore the sequence $x_i$ is periodic with period at most $16t\log\log(t+3)$.
\paragraph{}
We now turn to our proof of Conclusion (\ref{item:RankOfWindows2}). Since $(B_0,\ldots,B_n,0,\ldots,0) \in V^{\perp}$ and $N > |R|^{n+1} + 3(n+1)$, we may apply Lemma~\ref{lem:ExpressAsRoots} to write $x' = x_1 + \cdots + x_n$, where $(x_i(r))_{r=d}^{N-d} = (\alpha_i\rho^r_i)_{r=d}^{N-d}$, $\alpha_i \in \mathbb{C}$ and $\rho_i$ is a root of the polynomial 
\[ B_0 + B_1X + \cdots + B_nX^n = \left(\prod_{r = 1}^k \left( 1 - X^{r} \right)\right)^2 ,
\] for all $i\in [n]$. But this means, for each $i \in [l]$, that $\rho_i $ is a $q_i$th root of unity, for some $q_i \in [k]$. Thus each of $x_1,\ldots,x_n$ are periodic with period at most $k$. \qed
\paragraph{}
The following lemma makes explicit two important properties of the polynomials $S_k$. Firstly, that when an exponential polynomial $f$ is multiplied by $S_k$, $S_k$ will ``kill'' $\hat{f}(r)$ in the intervals where $\hat{f}(r)$ is periodic, with period at most $k$. Secondly, if $f$ has coefficients from a finite set, two multiplications by $S_k$ is essentially the same as one, with respect to this ``killing'' property. 

\begin{lemma} \label{lem:PropertiesOfSk}
For $k \in \mathbb{N}$ and $M,N \in \mathbb{Z}$, let $R \subseteq \mathbb{Q}$ be a finite set and $f$ be an exponential polynomial. 
\begin{enumerate}
\item \label{item:SkKills} If $(\hat{f}(r))_{r \in [N,M]}$ is periodic with period at most $k$ then $\widehat{S_kf}(r) = 0 $, for $r \in [N + \deg(S_k) ,M]$.
\item \label{item:indempotent} If $M - N > |R|^{2\deg(S_k)} + 6\deg(S_k)$ and $\hat{f}(r) \in R $, for all $r \in [N,M]$ then, $\widehat{S^2_kf}(r) = 0 $, for $r \in [N,M]$, implies that $\widehat{S_kf}(r) = 0 $, for $r \in [N + 3\deg(S_k),M - 2\deg(S_k)]$. 
\end{enumerate}
\end{lemma}
\begin{proof}
If $(\hat{f}(r))_{r \in [N,M]}$ is periodic with period $p \leq k$ we set $T  = S_k/(1 - e^{ip\theta})$ and notice that $\left( \widehat{Tf}(r) : r \in [N + \deg(T),M] \right)$ is periodic with period $p$. Thus 
\[ \widehat{S_kf}(r) = (T(1 - e^{ip\theta})f)^{\widehat{}}(r) = \widehat{Tf}(r) - \widehat{Tf}(r-p) = 0,
\] for all $r  \in [N + \deg(T) + p,M ] = [N + \deg(S_k),M ]$.
\paragraph{}
We shall deduce Conclusion~(\ref{item:indempotent}) of Lemma~\ref{lem:PropertiesOfSk} from Conclusion~(\ref{item:RankOfWindows2}) of Lemma~\ref{lem:RankOfWindows}. Write $n = \deg(S^2_k)$, express $S^2_k(\theta) = C_0 + C_1e^{i\theta} + \cdots + C_{n}e^{in\theta}$, for some $C_0,\ldots,C_{n} \in \mathbb{Q}$, and write $v = (C_0,\ldots,C_{n})$. The assumption of the lemma tells us that the vector $v$ lies in orthogonal complement to the space of $(n+1)$-windows of the sequence 
$\left(\hat{f}(r) : r \in [N,M] \right)$. Apply the second conclusion of Lemma \ref{lem:RankOfWindows} to find periodic sequences $(x_1(n))_{n\in [N+n,M-n]}, \ldots,  (x_l(n))_{n \in [N+n,M-n]}$, $l \in \mathbb{N}$, with period at most $k$, such that 
\[ \hat{f}(r) = x_1(r) + \cdots + x_l(r) ,
\] for all $r \in [N + n ,M - n]$. Here we have used the lower bound on $M-N$ to allow for the application of Lemma~\ref{lem:RankOfWindows}. It then follows from the first conclusion of the present lemma that $\widehat{S_kf}(r) = 0 $ for all $r \in [N + \deg(S_k) + n,M-n  ] = [N + 3\deg(S_k),M-2\deg(S_k)  ] $. \end{proof} 

We have now established the basic results that we shall need to prove Theorem~\ref{thm:CorWithPolyThm}, our result on the structure of exponential polynomials that correlate with a low-degree polynomial. 

\section{Reduction to The Structured Case } \label{sec:reduction}
The goal of this section is to prove Theorem~\ref{thm:CorWithPolyThm}. That is, we show that every exponential polynomial $f$ that ``correlates'' with a low-degree polynomial must have a special structure. It is easy to see how a theorem of this type connects with our main goal of finding roots of cosine polynomials: if $f$ has ``few'' zeros then there exists a low-degree polynomial $P$ for which $Pf \geq 0$ on $[0,2\pi]$, and so $\int Pf \geq \int |Pf|$. We make note of this construction here, although we shall use it only later.
\paragraph{}
If $f : [0,2\pi] \rightarrow \mathbb{R}$, is a continuous function, say that $\theta_0 \in [0,2\pi)$ is a \emph{sign change of} $f$ if $f(\theta_0) = 0$ and $f(\theta_0 + \varepsilon)f(\theta_0 - \varepsilon) \leq 0$ for all sufficiently small $\varepsilon \geq 0$. Given a continuous, \emph{even} function $f : \mathbb{R}/(2\pi\mathbb{Z})\rightarrow \mathbb{R}$ with finitely many (distinct) sign changes $\theta_1,\ldots,\theta_k \in [0,\pi]$, we define the \emph{companion polynomial} of $f$ to be \begin{equation} \label{equ:CompanionPoly} P(\theta) =  (-1)^p \prod_{i = 1}^{k} \left( \cos(\theta) - \cos(\theta_i) \right), \end{equation} where $p \in \{0,1\}$ is chosen so that $f(\theta)P(\theta) \geq 0$ for all $\theta \in [-\varepsilon,\varepsilon]$, for some $\varepsilon >0$. 

\begin{lemma} \label{prop:CompanionPoly} For $k \in \mathbb{N}$, let $f : \mathbb{R}/(2\pi\mathbb{Z}) \rightarrow \mathbb{R}$ be a continuous, even, real-valued function with $k$ distinct sign changes. The companion polynomial $P$ is such that 
\begin{enumerate}
\item $P$ is an exponential polynomial of degree $2k$ and is of the form $P(\theta) = \sum_{r = -k}^k \hat{P}(r)e^{ir\theta}$;
\item $P(\theta)f(\theta) \geq 0$, for all $\theta \in [0,2\pi]$.\qed \end{enumerate} \end{lemma}

We now turn to prove the core lemma (Lemma~\ref{lem:ReductionToStructure}) in the proof our ``structure theorem''. For our application to the problem of finding roots for cosine polynomials, we actually use a slightly different expression than the one outlined in Theorem~\ref{thm:CorWithPolyThm}. This form is described in the conclusion of Lemma~\ref{lem:ReductionToStructure2} and takes advantage of the extra symmetry we gain by assuming our polynomial is a cosine polynomial. 

\begin{lemma} \label{lem:ReductionToStructure}
For $\varepsilon >0$, $k,n \in \mathbb{N}$, let $R \subseteq \mathbb{Z}$ be a finite set and $f$ be a exponential polynomial with $\hat{f}(r) \in R$, for all $r \in \mathbb{Z}$. If there exists a non-zero exponential polynomial $P$ with $\deg(P) \leq k$ and such that  
\[ \left| \int Pf \right| \geq \varepsilon \int|Pf|, 
\] then there exists a non-zero exponential polynomial $Q$, with 
$\deg(Q) \leq 2^{12}(k\log_{(2)}(2k+3))^{2}$ for which 
\[ \log\log \left|\{ r \in \mathbb{Z} : \widehat{Qf}(r) \not= 0 \}\right| \leq  2^{14}k^2\log^2_{(2)}(2k+3)\log |R| + 8k\log M(R)   + \log \varepsilon^{-1} .\] 
\end{lemma}
\begin{proof}
Note that the statement of the theorem is trivial if $f$ is the zero polynomial and thus we may assume that there is some value of $r$ for which $\hat{f}(r) \not=0$.
Also note that $0 \in R$ as there must be some $r \in \mathbb{Z}$ for which $0 = \hat{f}(r) \in R$. Hence $|R| \geq 2$. We set $D = \lfloor 2^5k\log_{(2)}(2k+3) \rfloor$ and put $S(\theta) = S_D(\theta)$, as defined at (\ref{def:DefOfSk}). Observe that $\deg(S) \leq D^2$ and $|S|_{\infty} \leq 2^{D}$. Define \[ I = \frac{1}{2\pi}\left| \int Pf \right| \] and note that 
\[ I = \left| \sum_{r \in \mathbb{Z}} \overline{\hat{P}}(r)\hat{f}(r) \right| = \left| \sum_{r \in [-k,k]} \overline{\hat{P}}(r)\hat{f}(r) \right|, 
\] by using Parseval's identity along with the fact that $\hat{P}(r) = 0 $ for $r \not\in [-k,k]$. By our assumption on $P$, we have 
\begin{equation} \label{equ:absValInt} 2\pi I = \left| \int Pf \right| \geq \varepsilon \int |Pf|. 
\end{equation} We now obtain a lower bound on the right hand side integral. To do this we apply Theorem~\ref{thm:LittlewoodConj} to the polynomial
$PSf$ to obtain a function $h = h(PSf)$. Put $B = \{ r \in \mathbb{Z}: \widehat{PSf}(r) \not= 0 \}$ and express $B = \{ b_1 < b_2 < \cdots < b_m\}$, for some $m \in \mathbb{N}$. We have  
\begin{equation}\label{equ:applicationOfLittlewood} \int |Pf| \geq |S|^{-1}_{\infty} \int PSfh \geq (60|S|)^{-1}_{\infty}\sum_{i = 1}^m \frac{| \widehat{PSf}(b_i) |}{i},
\end{equation} where we have used the fact that $|h|_{\infty} \leq 1$ for the first inequality and Theorem \ref{thm:LittlewoodConj} for the second. The difficulty in analysing this sum comes from the fact that we have no direct control over the coefficients of $P$ and therefore little control over the individual quantities $|\widehat{PSf}(b_i) |$. However, we can get just enough control on the sum in Equation~\ref{equ:applicationOfLittlewood} for our purposes. In particular, we show that $| \widehat{PSf}(r) |$ cannot be small for too many consecutive values of $r$ without simply being zero. This means that not too many consecutive values of $\{|\widehat{PSf}(b_i)|\}_{i=1}^m$ can be small. 
\paragraph{}
To this end, let $R' = \{ \widehat{Sf}(r) \}_{r \in \mathbb{Z}}$ and notice that 
\[ R' \subseteq \left\lbrace a_1+ \cdots + a_{2^{D-1}} - a'_1 - \cdots - a'_{2^{D-1}}  : a_1,\ldots,a_{2^{D-1}},a'_1,\ldots,a'_{2^{D-1}} \in R \right\rbrace .
\] From this we conclude that $R'$ is a set of integers with $M(R') \leq 2^D M(R)$. It follows that $|R'| \leq 2^{D+1}M(R)+1$. Also set $G(R,k) = M(R')^{2k}(2k+1)^{k+1/2}$, for convenience in notation. We now prove a sequence of two central claims. Once we have established these, the proof of Lemma~\ref{lem:ReductionToStructure} will quickly follow. 
\begin{claim} \label{claim:BigOrZero}
If $M,N \in \mathbb{Z}$ are such that $M-N \geq |R|^{4D^2}M(R)^{2k+1} + 8D^2$ and  
\[ |\widehat{PSf}(r)| < I \cdot ((2k+1)M(R)G(R,k))^{-1}, 
\] for all $r \in [N,M]$, then $\widehat{PSf}(r) =0$ for all $r \in [N+5D^2, M - 5D^2]$. 
\end{claim}
\emph{Proof of claim: } First assume that the space of $(2k+1)$-windows of the sequence $(\widehat{Sf}(r))_{r \in [N,M]}$ has dimension $2k+1$. In this case we derive a contradiction. Write $P(\theta) = \sum_{r \in [-k,k]} \hat{P}(r)e^{ir\theta}$ and define a vector $x \in \mathbb{C}^{[-k,k]}$ to be 
$x = (\hat{P}(-k),\ldots, \hat{P}(0),\ldots,\hat{P}(k))$. Define $A$ to be a $(2k+1)\times (2k+1)$ matrix constructed by taking as rows $(2k+1)$ linearly independent $(2k+1)$-windows of $(\widehat{Sf}(r))_{r \in [N,M]}$. It follows that $A$ is an invertible matrix with entries in $R' \subseteq \mathbb{Z}$ and that $Ax = b$, where $b \in \mathbb{C}^{2k+1}$ is a vector with $|b|_{\infty} < I \cdot ((2k+1)M(R)G(R,k))^{-1}$. Now apply Lemma~\ref{lem:InversionBound} to deduce that $|x|_{\infty} < I/((2k+1)M(R))$. This leads to the contradiction
\[ I = \left| \sum_{r \in [-k,k]} \overline{\hat{P}}(r)\hat{f}(r) \right| \leq (2k+1)M(R)|x|_{\infty} < I.\]
\paragraph{}
Hence we may assume that the space of $(2k+1)$-windows of the sequence $(\widehat{Sf}(r))_{r \in [M,N]}$ has dimension less than $2k+1$.
This time our hope is to apply Lemma~\ref{lem:RankOfWindows}. First notice that
\[ M-N \geq |R|^{4D^2}M(R)^{(2k+1)} + 8D^2 > |R'|^{(2k+1)} + 3(2k + 1),\] where the first inequality holds by the assumption in the claim and the second inequality holds as  $ 2^{D+2}M(R) > 2^{D+1}M(R) + 1 \geq |R'| $, $|R| \geq 2$ and $2D > (2k+1)$. Also note that $\widehat{Sf}(r) \in R' \subseteq \mathbb{Q}$, for $r \in [N,M]$. Put $N' = N+(2k+1)$ and $M' = M-(2k+1)$. Now apply Lemma~\ref{lem:RankOfWindows} to find exponential polynomials $T_1(\theta),\ldots,T_l(\theta)$, for $l \in \mathbb{N}$, such that 1) $\widehat{T}_i(r)$ is supported on $[N',M']$ for each $i \in [l]$; 2) the $T_i$ are such that 
$(\widehat{T}_i(r))_{r \in [N',M']}$ is periodic with period at most $\lfloor 2^5k\log_{(2)}(2k+3)) \rfloor = D$, for all $i \in [l]$; and 3)
\[ \widehat{Sf}(r) = \sum_{i=1}^l \widehat{T}_i(r), 
\] for $r \in [N',M']$. By Conclusion~(\ref{item:SkKills}) in Lemma~\ref{lem:PropertiesOfSk}, it follows that 
\[ \widehat{S^2f}(r) = \sum_{i=1}^l \widehat{ST}_i(r) = 0, 
\] for all $r \in [N'+\deg(S),M'] $. By the lower bound on $M-N$, the inequality $\deg(S) \leq D^2$, and the fact that $M(R) \geq 1$, we have that 
\[ M' - N' - \deg(S) \geq |R|^{4D^2} + 8D^2 - \deg(S) - 2(2k+1) > |R|^{2\deg(S)} + 6\deg(S).\] Also $\hat{f}(r) \in R \subseteq \mathbb{Q}$, for all $r \in \mathbb{N}$ and so we may apply the second conclusion of Lemma~\ref{lem:PropertiesOfSk} to obtain
\[ \widehat{Sf}(r) = \sum_{i=1}^l \widehat{ST}_i(r) = 0 ,
\] for all $r \in [N'+ 4\deg(S),M'-2\deg(S)] $. This implies that $\widehat{PSf}(r) = 0 $ for all 
\[r \in [N' + 4\deg(S) + k,M'-k - 2\deg(S)] \supseteq [N+ 5D^2,M-5D^2]. \] This completes the proof of the claim. \qed
\paragraph{}

In the next claim, we use Claim~\ref{claim:BigOrZero} along with the bound on the sum at (\ref{equ:applicationOfLittlewood}) to derive an upper bound on set the size of the set $B = \{ r \in \mathbb{Z} : \widehat{PSf}(r) \not= 0 \} = \{ b_1 < b_2 < \cdots < b_m \}$.  

\begin{claim} \label{claim:BoundOnB} We have that
\[ \log |B| \leq 2^{10}\varepsilon^{-1}\cdot (2k+1)M(R)G(R,k)2^D|R|^{8D^2}M(R)^{2k+1}.\]
\end{claim}
\emph{Proof of claim:} 
Set $L = |R|^{8D^2}M(R)^{2k+1}$ and note that $L \geq |R|^{4D^2}M(R)^{2k+1} + 8D^2$, the quantity appearing in Claim~\ref{claim:BigOrZero}. We may assume that $|B| > 2^L$, otherwise we are done. 
Now, for $j \in [0,|B| -L]$, we claim that at least one of
\begin{equation} \label{equ:oneOfIfBig} |\widehat{PSf}(b_{j+1})|,|\widehat{PSf}(b_{j+2})|, \ldots, |\widehat{PSf}(b_{j+L})| 
\end{equation} is at least $\alpha = I/((2k+1)M(R)G(R,k))$. To see this, assume that $|\widehat{PSf}(b_{j+i})| < \alpha$, for each $i \in [L]$. As $\widehat{PSf}(r)$ is supported on $B$, it follows that $|\widehat{PSf}(r)| < \alpha$ for all $r \in [b_{j+1},b_{j+L}]$, an interval of length at least $L$. Hence we may apply Claim~\ref{claim:BigOrZero} to learn that $\widehat{PSf}(r) = 0 $ for all $r \in [b_{j+1} + 5D^2,b_{j+L} - 5D^2]$. So
\begin{align} L &= \left| B \cap [b_{j+1},b_{j+L}] \right| \leq 10D^2 + \left| B \cap [b_{j+1} + 5D^2, b_{j+L}- 5D^2] \right|   \\
 \label{equ:LatmostD} &= 10D^2 .\end{align} However, $L \geq 2^{8D^2} > 10D^2$, which contradicts (\ref{equ:LatmostD}). This implies that at least one of the quantities at (\ref{equ:oneOfIfBig}) is at least $\alpha$.
Recalling that we have set $m = |B|$ and that the logarithms are in base 2, it follows that 
\begin{align}
\label{equ:lowerboundsum} \sum_{i = 1}^m \frac{| \widehat{PSf}(b_i) |}{i} &\geq \sum_{j = 0}^{\lfloor (m-L)/L\rfloor} \sum_{i= 1}^L \frac{| \widehat{PSf}(b_{i+Lj}) |}{i + Lj} \\
&\geq \sum_{j = 1}^{\lfloor m/L\rfloor} \frac{\alpha}{jL} \\
&\geq \frac{(\alpha \log_e 2) \log(m/L -1) }{L} \\
\label{equ:boundOnLogSum} &\geq \frac{\alpha \log m  }{2L} \\ 
\label{equ:boundOnLogSum2} &= \frac{I\log |B|}{2(2k+1)M(R)G(R,k)L}. 
\end{align} To obtain the inequality at (\ref{equ:boundOnLogSum}) we have used the fact that $m > 2^L $ and $D \geq 16$ and thus $L \geq 2^{82^8}$. For this large $L$, we certainly have the the inequality $2^{L(1 - 1/(2\log_e 2))} > L$ and thus (\ref{equ:boundOnLogSum}). Putting the lower bound at (\ref{equ:boundOnLogSum2}) together with the bounds at (\ref{equ:absValInt}) and (\ref{equ:applicationOfLittlewood}) on the quantity at  (\ref{equ:lowerboundsum}), we arrive at the inequality
\[ 2\pi I \cdot (120\varepsilon^{-1}|S|_{\infty}(2k+1)M(R)G(R,k)L) \geq I \cdot \log |B|. \] Then apply the inequality $|S|_{\infty} \leq 2^{D}$, and the definition of $L$ to finish the proof of the claim. \qed 
\paragraph{}
From Claim~\ref{claim:BoundOnB}, we deduce Lemma~\ref{lem:ReductionToStructure}. 
In accordance with the statement of the Lemma, we choose $Q = PS$ and note that $\deg(PS) \leq 2k + 2D^2 \leq 2^{12}k^2\log^2_{(2)}(2k+3)$. 
All that remains is to bound $\left| \{ r \in \mathbb{Z} : \widehat{PSf}(r) \not= 0 \} \right| $. Of course, this set is exactly $B$. 
First observe that 
\[ G(R,k) = M(R')^{2k}(2k+1)^{k + 1/2} \leq 2^{2kD}M(R)^{2k}(2k+1)^{k+1/2},\] by using the bound $M(R') \leq 2^{D}M(R)$, obtained above. Thus 
\begin{align} 
\log G(R,k) &\leq 2kD + 2k\log M(R) . \\
&\leq (2^5 + 2)k^2 \log_{2}(2k+3) + 2k\log M(R). \end{align} 
Using this, along with Claim~\ref{claim:BoundOnB}, we have 
\begin{align*}
\log \log |B| &\leq 10 + \log 4k + \log M(R) + \log G(R,k)  + D + 8D^2\log |R| + (2k+1) \log M(R) + \log \varepsilon^{-1}\\
&\leq 2^{14}k^2\log^2_{(2)}(2k+3)\log |R| + 8k\log M(R)   + \log \varepsilon^{-1}.
\end{align*} This completes the proof of Lemma~\ref{lem:ReductionToStructure}.
\end{proof}
From Lemma~\ref{lem:ReductionToStructure} we may readily deduce our structural result on exponential polynomials, Theorem~\ref{thm:CorWithPolyThm}. 
Here we prove a slightly different form of this result that makes explicit use of the added symmetry gained from working with cosine polynomials. This is what we actually use for our task of finding zeros of cosine polynomials. The proof of Theorem~\ref{thm:CorWithPolyThm} follows from a very similar argument. 

\begin{lemma} \label{lem:ReductionToStructure2}
For $d,n,K \in \mathbb{N}$, let $R \subseteq \mathbb{Z}$ be a finite set and  
\[ f(\theta) = \sum_{r =0}^n C_r\cos(r\theta),
\] be a cosine polynomial, where $C_r \in R$, for all $r \in [0,n]$. If there is an exponential polynomial $Q$ of degree $d$ such that $\left|\{ r : \widehat{Qf}(r) \not= 0 \} \right| \leq K$ then there exists disjoint intervals $I_1,\ldots,I_l \subseteq \mathbb{N}$, $l \in [K]$, and cosine polynomials $f_1,\ldots,f_l,E$ so that 
\begin{enumerate}
\item \label{item:formOfFun}$f(\theta) = \sum_{i = 1}^l f_i(\theta) + E(\theta)$;
%\item \label{item:latMostK} $l \leq K$;
\item \label{item:FourierSupport} for each $i \in [l]$, $\hat{f}_i(r)$ is supported on $I_i \cup -I_i$;
\item \label{item:BoundOnError} $|E|_{\infty} \leq 8KM(R)(|R|^{d} + d)$;
\item \label{item:lemRtoS} for each $i \in [L]$ the sequence 
$(\hat{f}_i(r))_{r \in I_i}$ is periodic with period dividing $\left\lfloor 16d\log\log(d+3) \right\rfloor !$.
\end{enumerate}
\end{lemma}
\begin{proof}
For $L \in [K+1]$, we partition $\mathbb{N} \setminus \{ r : \widehat{Qf}(r) \not= 0 \} = J_1 \cup \cdots \cup J_L$, where $J_1,\ldots,J_L \subseteq \mathbb{N}$ are maximal disjoint intervals and $\max J_i < \min J_{i+1}$, $i \in [L-1]$. Note that $J_L$ must be the unique interval of infinite length and that $\hat{f}(r) = 0$ for $r \in J_L$, as $f$ is an exponential polynomial. Let $J'_1,\ldots, J'_l$ be an enumeration of $\{ J_i :  |R|^d + 3d < |J_i| < \infty  \}$, where we have $l \in [L-1]$. For $i \in [l]$, let $I^+_i$ be the first $d$ and last $d$ elements of the interval $J'_i$ and let $I_i = J'_i \setminus I^+_i$.

Define \[ S = \{0\} \cup \{ r \in \mathbb{N}: \widehat{Qf}(r) \not= 0 \} \cup \bigcup_{i=1}^l I^+_i \cup \bigcup_{i : |J_i| \leq |R|^d + 3d } J_i . \] We have $|S| \leq 1 + K+ 2dK + K(|R|^d + 3d) \leq 8K(|R|^d + d)$. Then
\[ f(\theta) = \sum_{r\in S} C_r\cos(r\theta) + \sum_{i=1}^l \sum_{r \in I_i } C_r\cos(r\theta) ,
\] as the sets $S \cup I_1 \cup \cdots \cup I_l$ form a partition of $[0,n]$. We define the ``error term'' $E$ to be 
\[ E(\theta) = \sum_{r \in S} C_r \cos(r\theta)
\] and note that $|E(\theta)|_{\infty} \leq |S|M(R) \leq 8KM(R)(|R|^d + d)$. For $i \in [l]$, set
\[ f_i(\theta) = \sum_{r \in I_i } C_r\cos(r\theta) = \sum_{r \in I_i \cup -I_i} \hat{f}(r)e^{ir\theta} .
\] All that remains is to check that the conclusions of the Lemma hold for these choices. Conclusion~(\ref{item:formOfFun}), (\ref{item:FourierSupport}) and (\ref{item:BoundOnError}) follow directly from the construction. Conclusion~(\ref{item:lemRtoS}) follows from an application of Lemma~\ref{lem:RankOfWindows}. Fix $i \in [l]$. Since $\widehat{Qf}(r) = 0 $ for $r \in J'_i$, it follows that the the $d+1$ windows of the sequence $(\hat{f}_i(r))_{r \in J'_i}$ have rank at most $d$. Since $|J'_i| > |R|^d + 3d$, and $\hat{f}_i$ are rational numbers taking values in $R$, we may apply Lemma~\ref{lem:RankOfWindows} to find $m \in \mathbb{N}$ and sequences $(x_1(r))_{r \in I_i},\ldots , (x_m(r))_{r \in I_i}$, which are periodic with period at most $\lfloor 16d\log\log(d+3) \rfloor $ and  $\hat{f}_i(r) = x_1(r) + \cdots + x_m(r)$, for all $r \in I_i$. This implies that $(\hat{f}_i(r))_{r \in I_i}$ is periodic with period dividing $\lfloor 16d\log\log(d+3) \rfloor!$. 
\end{proof}

\section{The Structured Case} \label{sec:theStructuredCase}
In this section we prove that cosine polynomials which satisfy the strong structural property in the conclusion of Lemma~\ref{lem:ReductionToStructure2} must have many zeros. We prove our results here in two steps. 

Let us first define the function $D_n(\theta) = \sum_{r = 0}^n \cos(r\theta)$, for $n \in \mathbb{N}$. In our first step, we show that every function $f$ that is a linear combination of a bounded number of functions of the type $D_n$ must have many zeros. We then show that there is a low degree polynomial $F$ so that every cosine polynomial $f$ of the form in Lemma~\ref{lem:ReductionToStructure2}, is such that $Ff$ is the sum of a bounded number of the functions $D_n$. It will then follow that $Ff$ has many roots. But since $F$ has low degree, it must be that $f$ itself has many zeros . 

%The obstacle here is error term $E$, about which we have assumed little except that it is small. We shall see that if  $|f(0)| \gg |E|_{\infty}$ the ``rest'' of the polynomial actually has fluctuations about the origin that are too large to be ``cancelled'' by $E$. Thus we conclude that $f$ has many zeros. Although, in what follows we take a ``global'' approach in the sense that we only work with averages over $[0,2\pi]$. So this point of view is not as salient in our approach.In this case the challenge comes from the ``error term'' $E(\theta)$, about which we know little, except that it is ``small''. 
\paragraph{}
For our first step we use a basic and convenient fact about the function $D_n$ (See \cite{Zyg}, for example). 
 
\begin{lemma} \label{fact:DKernelBound} For $n \in \mathbb{N}$ and every interval $J \subseteq \mathbb{R}/(2\pi\mathbb{Z})$ we have 
\[ \left| \int_J D_n(\theta)  \right| \leq 10.\] \qed
\end{lemma}

We now show that every (bounded) linear combination of the functions $D_n$ has many roots. For this, we use a simple averaging argument:  Lemma~\ref{fact:DKernelBound}, tells us that $g$ has small average value on every interval, while Theorem~\ref{thm:LittlewoodConj} tells us that $\int |g|$ is large. Therefore $g$ must have many zeros. 

\begin{lemma} \label{lem:ZerosForSumsOfD} For $l,M,P \in \mathbb{N}$, let $A_1,\ldots,A_l$ be real numbers which are either $0$ or satisfy $1/P \leq |A_i| \leq 2M$ and let $g$ be a cosine polynomial of the form 
\[ g(\theta) = \sum_{i=1}^l A_i\sum_{r \in I_i } \cos(r\theta) + E(\theta), 
\] where $I_1,\ldots,I_l \subseteq \mathbb{N}$ are finite intervals which satisfy $\max I_i < \min I_{i+1}$, $i \in [l-1]$, and $E$ is a real-valued function. 
Set $Z = (|g(0)| - |E|_{\infty})/M$. If $Z > 0 $ then $g(\theta)$ has at least \begin{equation} \label{equ:BoundInDzeros} \frac{\log(Z)}{240P(20lM + \pi|E|_{\infty})}-1 \end{equation} sign-changes in $[0,2\pi]$.
\end{lemma}
\begin{proof}
Suppose that $g$ has exactly $k$ sign changes and that $J_1 \cup \cdots \cup J_k = \mathbb{R}/(2\pi\mathbb{Z})$ are the maximal intervals on which the sign of $g$ is constant. We have that 
\begin{equation} \label{equ:fewRootsSimplify}  \int |g| =  \sum_{i = 1}^k \left|\int_{J_i}g(\theta)\right| .
\end{equation} We use Theorem~\ref{thm:LittlewoodConj} (the solution to Littlewood's $L_1$-norm Conjecture) to estimate the left-hand-side. Define integer intervals $\tilde{I}_1,\ldots, \tilde{I}_{l'} \subseteq\mathbb{N}$ and $B_1,\ldots,B_{l'} \in \mathbb{R}$ to be such that $l'\in [l]$ and $\tilde{I}_i$ is the $i$th interval, $I_j$ say, among $I_l,I_{l-1} \ldots,I_1$ (in this order) for which $A_j \not= 0 $. We then set $B_i = A_j$. Letting $t_0 = 0$ and $t_i = |\tilde{I}_1| + \cdots + |\tilde{I}_i|$, for $i \in [l']$, we have
\begin{align}
\int |g| &\geq \sum_{i = 0}^{l'-1}\frac{|B_i|}{120} \sum_{r = t_i+1}^{t_{i+1}} \frac{1}{r} - 2\pi|E|_{\infty} \\
\label{equ:LowerboundStrcLemma} &\geq \frac{\log t_{l'} }{120P}- 2\pi|E|_{\infty}.
\end{align}
We now obtain a lower bound on $t_{l'}$ in terms of $|g(0)|$. By the definition of $g$, we have 
\begin{align*} |g(0)|/(2M) &\leq \sum_{i=1}^l \sum_{r \in I_i} |A_i|/(2M) + |E|_{\infty}/(2M) \\
&\leq |\tilde{I_1}| + \cdots + |\tilde{I_{l'}}| + |E|_{\infty}/(2M) \\
&\leq t_{l'} + |E|_{\infty}/(2M).
\end{align*} Thus $(|g(0)| - |E|_{\infty})/(2M) \leq t_{l'}$.
We now turn to estimate the right hand side of (\ref{equ:fewRootsSimplify}) from above. Start by writing $a_i = \min I_i - 1$ and $b_i = \max I_i$, for each $i \in [l]$. Recall that we have defined $D_n(\theta) = \sum_{r =  0 }^n \cos(r\theta) $. Thus we have  
\[ g(\theta) = \sum_{i =1}^l A_i (D_{b_i}(\theta) - D_{a_i}(\theta)) + E(\theta). \] Then apply Lemma~\ref{fact:DKernelBound} to learn that
\[ \left| \int_{J_j} g \right| \leq  20\sum_{i = 1}^l |A_i| + 2\pi |E|_{\infty}, \] for all $j \in [k]$. Therefore the right hand side (\ref{equ:fewRootsSimplify}) is bounded above by \[ 20k\sum_{i = 1}^l |A_i| + 2\pi k|E|_{\infty} \leq k(40lM + 2\pi |E|_{\infty}).  \] Putting this together with the bound at (\ref{equ:LowerboundStrcLemma}), we obtain a lower bound on $k$, the number of roots of $g$, as in the statement of the Lemma.  
\end{proof}

We now turn to prove that polynomials that are as in the conclusion to Lemma~\ref{lem:ReductionToStructure2} have many zeros. To do this, we show that we can massage a polynomial of the form in Lemma~\ref{lem:ReductionToStructure2} into the form of a polynomial in Lemma~\ref{lem:ZerosForSumsOfD} at the cost of introducing only a few more zeros.  

\begin{lemma} \label{lem:MainStructureLemma}
For $l,K,M \in \mathbb{N}$, let $I_1,\ldots,I_l \subseteq \mathbb{N}$ be finite, disjoint intervals. Let $f, f_1,\ldots,f_l,E$ be cosine polynomials such that
\[ f(\theta) = \sum_{i = 1}^l f_i(\theta) + E(\theta), 
\] and the following hold. 
\begin{enumerate}
\item For each $i \in [l]$, the sequence $(\hat{f}_i(r))_{r \in \mathbb{Z}}$ is supported on $-I_i \cup I_i$ and $(\hat{f}_i(r))_{r \in I_i}$ is periodic with period dividing $P$. 
\item For all $i \in [l]$ and $r \in \mathbb{Z}$ we have $2\hat{f}_i(r) \in \mathbb{Z}$ and $|\hat{f}_i(r)| \leq M$.
\item We have $|E|_{\infty} \leq K$. 
\end{enumerate} If $Y = |f(0)| - 4Pl + K/M  > 0$ then $f$ has at least $(2^{14}P^2Ml + 2^{10}KP))^{-1}\log(Y) - 2P-1$ roots in $[0,2\pi]$. 
\end{lemma}
\begin{proof}
Define the cosine polynomial 
\[ F(\theta) =  \frac{1}{P}\Big( 1 + \cos(\theta) + \cdots + \cos((P-1)\theta) \Big)
\] and note that $\deg(F) = 2(P-1)$ and $|F|_{\infty} \leq 1$. 
For $i \in [l]$, let us say that the sequence $(\hat{f}_i(r))_{r \in I_i}$ is periodic with period $p_i$, which divides $P$. Write $I_i = [a_i,b_i]$ and set 
\begin{equation}\label{equ:DefOfAi}
A_i =  2/p_i\sum_{r=a_i}^{a_i + p_i-1} \hat{f}_i(r). \end{equation} Due to the periodicity of $\hat{f}_i$ on $I_i$, for any $t$ such that $t,t+p_i-1 \in I_i$, we have
\[ A_i =  2/p_i\sum_{r = t}^{t + p_i-1}\hat{f}_i(r).\] We observe the following.
\begin{claim}
There exist intervals of positive integers $I'_1 \subseteq I_1 , \ldots, I'_l \subseteq I_l $ with
$|I'_i| \geq |I_i| - 2P$, $i \in [l]$, so that 
\[ Ff(\theta) = \sum_{i = 1}^l A_i \sum_{r \in I''_i} \cos(r\theta) + \tilde{E}(\theta),\]
where $\tilde{E}$ is a cosine polynomial with $|\tilde{E}(\theta)|_{\infty} \leq K + 4PMl$.
\end{claim} 
\emph{Proof of Claim :}
For each $i \in [l]$, define $I'_i$ to be $I_i$ with the first and last $P$ elements removed. Hence $|I'_i| \geq |I_i| - 2P$. We now use the identity $2\cos(a\theta)\cos(b\theta) = \cos((a+b)\theta) + \cos((a-b)\theta)$ to write the product $Ff_i$ as a sum of cosines. In particular,
\begin{align*} Ff_i(\theta) &= \left( \frac{1}{P}\sum_{r=0}^{P-1} \cos(r\theta) \right) \left( \sum_{r \in I_i } 2\hat{f}_i(r)\cos(r\theta) \right) \\
&= \sum_{r \in I'_i }\left(  2\hat{f}_i(r)/P + \sum_{a = 1}^{P-1}\hat{f}_i(r+a)/P + \sum_{a = 1}^{P-1}\hat{f}_i(r-a)/P \right) \cos(r\theta) + E_i(\theta) \\
&= \sum_{r \in I'_i }\left(  \frac{2P}{p_i}\frac{1}{P}\sum_{a = 0}^{p_i-1}\hat{f}_i(r+a) \right) \cos(r\theta) + E_i(\theta) \\
&= A_i\sum_{r \in I'_i }\cos(r\theta) + E_i(\theta).
\end{align*} Here, $E_i$ is a function of the shape $E_i(\theta) = \sum_{r \in I_i\setminus I'_i} D_r\cos(r\theta)$, where $D_r$ is a real number satisfying $|D_r| \leq 2M$, for all $r \in I_i'\setminus I_i$. Therefore $|E_i|_{\infty} \leq 4PM$. To finish the proof of the claim, we simply write $\tilde{E} = \sum_{i=1}^l E_i + E $ and observe that
\begin{align*}
Ff &=  Ff_1 + \cdots + Ff_l + FE \\
&=  \sum_{i=1}^l A_i\sum_{r \in I_i' } \cos(r\theta)  + \sum_{i=1}^l E_i + FE\\
&=  \sum_{i=1}^l A_i\sum_{r \in I_i' } \cos(r\theta) + \tilde{E}. 
\end{align*} Let us point out that $\tilde{E}$ is a real-valued function with $ |\tilde{E}|_{\infty} \leq 4PMl + K $. \qed

We now prepare to apply Lemma~\ref{lem:ZerosForSumsOfD} and finish the proof of Lemma~\ref{lem:MainStructureLemma}. From (\ref{equ:DefOfAi}) we have $|A_i| \leq 2M$, for all $i \in [l]$, and since $2\hat{f_i}(r) \in \mathbb{Z}$, we have that $|A_i| \geq 1/p_i \geq 1/P$, whenever $A_i \not= 0 $. So set $g(\theta) = F(\theta)f(\theta)$ and $Z = (|f(0)| - |\tilde{E}|_{\infty})/M \geq |f(0)| - 8P^2l + K/M $. Note that $g(0) = f(0)$ and apply Lemma~\ref{lem:ZerosForSumsOfD} to learn that $g$ has at least
\begin{align*}
s = \frac{\log(Z)}{240P(20lM + \pi|\tilde{E}|_{\infty})} -1 &\geq  \frac{\log(Z)}{240P(20lM + 2^4PMl + 4K)} -1  \\
&\geq \frac{\log(Z)}{2^{14}P^2Ml + 2^{10}PK} -1
\end{align*} roots. Since $g = Ff$ and $F$ is an exponential polynomial of degree $ \leq 2P$, it follows that $f$ has at least $s - 2P$ roots. This completes the proof of the Lemma. 
\end{proof}

Now we have established all of the ingredients that shall go into the proof of Theorem~\ref{thm:main2}. All that remains is to string these results together to obtain a general bound. 

\section{Proofs of Theorems~\ref{thm:main} and \ref{thm:main2}} \label{sec:ProofsOfThms}
\paragraph{} To complete the proof of Theorems~\ref{thm:main} and \ref{thm:main2}, we need only to sequentially evoke our various lemmas, while keeping track of our parameters. We keep track of the explicit constants here, unlike in the statement of the theorem, in case there is any interest for specific applications. 

\emph{Proof of Theorems \ref{thm:main} and \ref{thm:main2} : } 
Let $f$ be as in the statement of Theorem~\ref{thm:main2}. Set $N = |f(0)|$. Notice that it suffices to prove the theorem for polynomials with coefficients in $\mathbb{Z}$, so we assume that $R \subseteq \mathbb{Z}$. For a contradiction, suppose that $f$ has $k$ roots, where $k$ is less than 
\[ X =  \left\lfloor \frac{\log_{(3)} N - 8\log 2M(R) \cdot \log_{(3)}^{1/2} N}{2^{17} \log (|R|+1) \cdot\log^2_{(5) }N }   \right\rfloor^{1/2} .
\] Note $X \leq (2^{-17}\log_{(3)} N )^{1/2}$. Also, we may assume that $|R| \leq \log_{(2)} N$, $M(R) \leq \log\log N$, and that $N \geq 2^{2^{2^{2^{17}}}}$, otherwise the statement of the theorem is trivial. In what follows, we use this lower bound on $N$ repeatedly and implicitly.

Now let $P$ be the companion polynomial of $f$, as defined at (\ref{equ:CompanionPoly}). From Lemma~\ref{prop:CompanionPoly} we know that $Pf \geq 0$ and therefore we have $\int Pf \geq \int |Pf|$. Now define $h = 2f$  so that $\hat{h}(r) \in R'$, where $R'$ is some set with $R' \subseteq \mathbb{Z}$, $M(R') \leq 2M(R)$ and $|R'| \leq |R|+1$. We apply Lemma~\ref{lem:ReductionToStructure} to $h$ to find a non-zero exponential polynomial $Q$ such that
\[ D = \deg(Q) < 2^{12}(k\log\log(4k+3))^{2} \leq \log_{(3)} N, \]  and 
\[ \left|\{ r \in \mathbb{Z} :  \widehat{Qh}(r) \not= 0 \}\right| \leq (\log N)^{1/2}.
\] Of course $\widehat{Qh}(r) = 0$ if and only if $\widehat{Qf}(r) \not= 0$ and so we may apply Lemma~\ref{lem:ReductionToStructure2} to write $f$ in the form 
\[ f(\theta) = \sum_{i = 1}^l f_i(\theta) + E(\theta),\]
for some $l \leq \log^{1/2} N$, where $f_1,\ldots,f_l,E$ are cosine polynomials, and there exist finite, disjoint intervals $I_1,\ldots,I_l$ so that, for $i\in [l]$, $\hat{f}_i(r)$ is supported on $I_i \cup -I_i$ and the sequence $(\hat{f}_i(r))_{r \in I_i}$ is periodic with period at most 
\begin{align*}
 \lfloor 16D\log_{(2)} (D+3) \rfloor! &\leq 2^{16D\log_{(2)} (D+3)\log(16D\log_{(2)} (D+3))} \\
&\leq  2^{16(\log_{(3)} N)^4} 
\end{align*} We also have 
\begin{align*}
|E|_{\infty} &\leq 8\left|\{ r : \widehat{Qf}(r) \not= 0 \}\right| M(R)(|R|^D + D) \\
&\leq 8(\log^{1/2} N )(\log_{(2)} N )2^{2D\log|R|}  \\
&\leq (\log^{1/2} N) \cdot 2^{16(\log_{(3)} N)^4}  .
\end{align*} 
We now apply Lemma~\ref{lem:MainStructureLemma} with parameters $l \leq \log^{1/2} N$, $P =  2^{16(\log_{(3)} N)^4} $, $K = (\log^{1/2} N) 2^{16(\log_{(3)} N)^4}$ and $M = M(R) \leq \log\log N$. Let $Y$ be as in the statement of Lemma~\ref{lem:MainStructureLemma} and note that 
\[ Y = |f(0)| - 4Pl + K/M \geq N/2 . \] We also have
\[ 2^{14}P^2Ml + 2^{10}KP \leq 2^{2^{10}(\log_{(3)} N)^4}\log^{1/2} N    .
\] Thus the number of zeros of $f$ is at least
\begin{align*}
\frac{\log N/2}{2^{14}P^2Ml + 2^{10}KP} - 2P - 1 &\geq \frac{\log^{1/2} N }{2^{2^{10}(\log_{(3)} N)^4}} - 2^{32(\log_{(3)} N)^4} \\
&\geq \frac{\log^{1/2} N }{2^{2^{11}(\log_{(3)} N)^4}} \\
&\geq \log_{(3)}^{1/2} N   \\
&> X 
\end{align*}
But this contradicts our assumption that the number of roots of $f$ was at most $X$. This completes the proof.  \qed
\section{Acknowledgements}
I should like to thank B\'{e}la Bollob\'{a}s for direction and comments; Paul Balister, Tam\'{a}s Erd\'{e}lyi, Robert Morris and Kamil Popielarz for carefully reading and commenting on an earlier draft of this paper; and Tomas Ju\v{s}kevi\v{c}ius for introducing me to the problems of Littlewood and for many related conversations. Part of this research was conducted while visiting Cambridge University. I am grateful for the hospitality of the combinatorics group and Trinity College, Cambridge.

\Addresses
\end{document}